\documentclass[12pt]{amsart}

\usepackage{amssymb}
\usepackage{bm}
\usepackage{graphicx}
\usepackage{psfrag}
\usepackage{color}
\usepackage{url}

\newcommand{\excise}[1]{}%{$\star$\textsc{#1}$\star$}

\newtheorem{thm}{Theorem}[section]
\newtheorem{lemma}[thm]{Lemma}

\newtheorem{cor}[thm]{Corollary}
\newtheorem{prop}[thm]{Proposition}
\newtheorem{conj}[thm]{Conjecture}

\newtheorem{prob}[thm]{Problem}

\theoremstyle{definition}
\newtheorem{example}[thm]{Example}
\newtheorem{remark}[thm]{Remark}
\newtheorem{defn}[thm]{Definition}

\numberwithin{equation}{section}

\newcommand{\ring}[1]{\ensuremath{\mathbb{#1}}}

\renewcommand\>{\rangle}
\newcommand\<{\langle}

\newcommand\NN{\ring{N}}

\DeclareMathOperator\image{Im} % Hom
\DeclareMathOperator\lcm{lcm} % lcm
\DeclareMathOperator\Ap{Ap} % Apery set
\DeclareMathOperator\supp{supp} % support
\DeclareMathOperator\bul{bul} % bullets
\DeclareMathOperator\mbul{mbul} % maximal bullets

\begin{document}%%%%%%%%%%%%%%%%%%%%%%%%%%%%%%%%%%%%%%%%%%%%%%%%%%%%%%%%
%%%%%%%%%%%%%%%%%%%%%%%%%%%%%%%%%%%%%%%%%%%%%%%%%%%%%%%%%%%%%%%%%%%%%%%%

\mbox{}
%\vspace{-2ex}%-1.1743pt}
\title{On the Linearity of $\omega$-primality in~Numerical~Monoids\qquad}
\author{Christopher O'Neill}
\address{Mathematics Department\\Duke University\\Durham, NC 27708}
\email{musicman@math.duke.edu}
% \email{www.math.duke.edu/\~{\hspace{-.3ex}}musicman}
\author{Roberto Pelayo}
\address{Mathematics Department\\University of Hawai`i at Hilo\\Hilo, HI 96720}
\email{robertop@hawaii.edu}

%\makeatletter
%  \@namedef{subjclassname@2010}{\textup{2010} Mathematics Subject Classification}
%\makeatother
%\subjclass[2010]{Primary: 20M14, 05E40, 20M25; Secondary: 20M30,
%20M13, 05E15, 13F99, 13C05, 13A02, 13P99, 68W30, 14M25}
% Primary:
% 20M14 Commutative semigroups
% 05E40 Combinatorial aspects of commutative algebra
% 20M25 Semigroup rings, multiplicative semigroups of rings [See also 16S36, 16Y60]
% Secondary:
% 20M30 Representation of semigroups; actions of semigroups on sets
% 20M13 Arithmetic theory of monoids
% 05E15 Combinatorial aspects of groups and algebras
% 13F99 Arithmetic rings and other special rings--none of the above but in this section
% 13C05 Structure, classification theorems (theory of modules and ideals)
% 13A02 Graded rings [See also 16W50]
% 13P99 None of the above, but in this section (Computational aspects and applications)
% 68W30 Symbolic computation and algebraic computation
% 14M25 Toric varieties, Newton polyhedra [See also 52B20]
% Tertiary (didn't make the cut):
% 13P25 Applications of commutative algebra (to stat, control theory, optimization,...)
% 20M12 Ideal theory (of semigroups)
% 16W50 Graded rings and modules
% http://msc2010.org/MSC-2010-server.html
% EM: verify all MSC codes, because I've changed ``2000'' to ``2010''
\date{\today}

\begin{abstract}
\hspace{-2.05032pt}
In an atomic, cancellative, commutative monoid, the $\omega$-value measures how far an element is from being prime. In numerical monoids, we show that this invariant exhibits eventual quasilinearity (i.e., periodic linearity). We apply this result to describe the asymptotic behavior of the $\omega$-function for a general numerical monoid and give an explicit formula when the monoid has embedding dimension $2$.  
\end{abstract}
\maketitle

% \setcounter{tocdepth}{1}
% \tableofcontents

%%%%%%%%%%%%%%%%%%%%%%%%%%%%%%%%%%%%%%%%%%%%%%%%%%%%%%%%%%%%%%%%%%%%%%%%%
\section{Introduction}\label{s:intro}%%%%%%%%%%%%%%%%%%%%%%%%%%%%%%%%%%%%
%raggedbottom%%%%%%%%%%%%%%%%%%%%%%%%%%%%%%%%%%%%%%%%%%%%%%%%%%%%%%%%%%%%

Recent years have seen an intensive study of the factorization properties in monoids and integral domains.  Numerical monoids (i.e., co-finite, additive submonoids of $\mathbb N$) provide an excellent venue to explore various measurements of non-unique factorization, including elasticity \cite{elasticity}, delta sets \cite{delta}, and catenary degrees \cite{catenarytame}. See \cite{numerical} for precise definitions. In \cite{tame}, Geroldinger and Hassler define the $\omega$-primality of an element in a cancellative, commutative monoid, which measures how far an element is from being prime.  Their definition is provided below.

\begin{defn}\label{d:omega}
Let $M$ be a cancellative, commutative atomtic monoid with set of units $M^{\times}$ and set of irreducibles $\mathcal A(M)$.  For each $x \in M$, define $\omega(x) = m$ if $m$ is the smallest positive integer with the property that whenever $x \mid a_1\cdots a_t,$ with $a_i \in \mathcal A(M)$, there is a $T \subset \{1,2, \ldots, t\}$ with $|T| \leq m$ such that $x \mid \prod_{k \in T} a_k.$  If no such $m$ exists, then set $\omega(x) = \infty$.  For each $x \in M^{\times}$, we define $\omega(x) = 0$.  
\end{defn}

Soon after, Anderson and Chapman developed much of the foundational theory for $\omega$-primality in integral domains \cite{prime, bounding}.  They proved that, when applied to elements in numerical monoids, $\omega$-primality becomes fairly well-behaved and tractable.  Furthermore, in \cite{andalg}, they develop an algorithm to compute $\omega$-values for any numerical monoid.  A variant of this function is now included in the standard \texttt{numericalsgps} package in \texttt{GAP} \cite{numericalsgps}.  This paper applies their algorithm to give a description of values of the $\omega$-function for elements in an embedding dimension $2$ numerical monoid.  In embedding dimension $3$, the authors of \cite{interval} provide a closed form for the $\omega$-function on the generators of a numerical monoid generated by an interval.

In this paper, we develop results on the eventual quasi-linearity of $\omega(n)$, viewed as a function $\omega: \Gamma \to \mathbb N$, for any numerical monoid $\Gamma$.  That is, if $n_1$ is the smallest generator of $\Gamma$, we prove that for large enough $n \in \Gamma$,  $\omega(n) = \frac{1}{n_1}n + a(n)$, where $a(n)$ is periodic with period dividing $n_1$.  In the case where $\Gamma$ has embedding dimension $2$, we apply these results to obtain an explicit formula for $\omega(n)$, which provides a more global description for $\omega$-values for large $n$ than the one given in \cite{andalg}.  

In Section~\ref{s:background}, we provide the reader with background and definitions related to $\omega$-primality.  In Section~\ref{s:quasi}, we prove the main result of our paper (Theorem~\ref{t:quasi}) using the concept of a cover map (Definition~\ref{d:covermap}), and use it to provide an asymptotic description of the $\omega$-function (Corollary~\ref{c:growth}).  In Section~\ref{s:2gen}, we provide an explicit form for $\omega(n)$ for $n$ sufficiently large when $\Gamma$ has embedding dimension 2 (Theorem~\ref{t:2omega}).  Lastly, in Section~\ref{s:futurework}, we give some open questions, including possible connections to Hilbert functions in combinatorial commutative algebra (Problem~\ref{p:hilbert}).  

% terminology differences from algorithm paper?

%%%%%%%%%%%%%%%%%%%%%%%%%%%%%%%%%%%%%%%%%%%%%%%%%%%%%%%%%%%%%%%%%%%%%%%%%
\section{Background}\label{s:background}%%%%%%%%%%%%%%%%%%%%%%%%%%%%%%%%%
%raggedbottom%%%%%%%%%%%%%%%%%%%%%%%%%%%%%%%%%%%%%%%%%%%%%%%%%%%%%%%%%%%%

In this section, we will build the basic machinery to compute $\omega$-primality for elements of numerical monoids.  In what follows, $\NN$ denotes the set of non-negative integers.

\begin{defn}\label{d:factor}
Let $\Gamma$ be a numerical monoid with minimal generating set $\{n_1, \ldots, n_k\}$, 
and fix $n \in \Gamma$.  A vector $\vec a = (a_1, a_2, \ldots, a_k) \in \NN^k$ is a \emph{factorization} of $n$ if $n = \sum_{i=1}^k a_in_i$.  
The \emph{length} of the factorization $\vec a$, denoted $|\vec a|$, is given by $\sum_{i=1}^k a_i$.  
The \emph{support} of a factorization $\vec a$, denoted $\supp(\vec a)$, is given by $\supp(\vec a) = \{n_j : a_j > 0\}$.  
\end{defn}

% The factorizations of a numerical monoid $\Gamma$ define 
% a monoid homomorphism $\NN^k \to M$ which sends $\vec a$ to $\sum_{i=1}^k a_im_i$, 
% and the preimage of any $m \in M$ is its set of factorizations.  

For a numerical monoid $\Gamma$, we provide a more pertinent definition 
of $\omega$-primality in terms of the minimal generators, 
using the additive structure of $\Gamma$ as a submonoid of~$\NN$.  

\begin{defn}\label{d:omega_num}
Let $\Gamma = \langle n_1, \ldots, n_k\rangle$ be a numerical monoid with irreducible elements $n_1, \ldots, n_k$.  For $n \in \Gamma$, define $\omega(n) = m$ if $m$ is the smallest positive integer with the property that whenever $\vec a \in \NN^k$ satisfies $\sum_{i=1}^k a_in_i - n \in \Gamma$ with $|\vec a| > m$, there exist a $\vec b \in \NN^k$ with $b_i \le a_i$ for each $i$ such that $\sum_{i=1}^k  b_i n_i - n \in \Gamma$ and $|\vec b| \le m$.  
\end{defn}

The notion of a bullet was first introduced in \cite{interval}.  
Throughout this paper, we use bullets extensively to study the $\omega$-function of numerical monoids.  
Definition~\ref{d:bullet} gives the definition of a bullet, 
and Lemma~\ref{l:bullet} justifies using bullets to study the $\omega$-function.  

\begin{defn}\label{d:bullet}
Fix a numerical monoid $\Gamma = \<n_1, \ldots, n_k\>$ and $n \in \Gamma$.  We say $\vec a \in \NN^k$ is a \emph{bullet} for $n$ if $(\sum_{i=1}^k a_in_i) - n \in \Gamma$ and $(\sum_{i=1}^k a_in_i - n_j) - n \notin \Gamma$ whenever $a_j > 0$.  A bullet $\vec a$ for $n$ is \emph{maximal} if $|\vec a| = \omega(n)$.  
Let $\bul(n)$ (resp. $\mbul(n)$) denote the set of bullets (resp. maximal bullets) of $n$, and let 
$$\bul_j(n) = \{\vec a \in \bul(n) : n_j \in \supp(\vec a)\} \subset \bul(n)$$
for $j \le k$.  
\end{defn}

\begin{lemma}\label{l:bullet}
There exists a maximal bullet for each $n \in \Gamma$.  In particular, $\omega(n) < \infty$ for all $n \in \Gamma$.
\end{lemma}

\begin{proof}
It suffices to show that $\omega(n)$ is finite.  
Indeed, if there were no bullets for $n$ of length $\omega(n)$, 
it would contradict the minimality of $\omega(n)$.  

Fix a factorization $\vec a$ for $n$, 
and fix $\vec b$ with 
$|\vec b| \ge |\vec a| \cdot \sum_{i = 1}^k n_i.$
We must have $b_j > a_j$ for some $j$, 
so if $\sum_{i=1}^k b_in_i - n \in \Gamma$, 
we have $\sum_{i=1}^k b_in_i - n_j - n \in \Gamma$,
meaning $\vec b$ is not a bullet for $n$.  
Thus, $\omega(n) < |\vec a| \cdot \sum_{i = 1}^k n_i$.  
\end{proof}

The main result of this paper is Theorem~\ref{t:quasi}, which states that 
for any numerical monoid $\Gamma$, $\omega(n)$ is quasilinear 
for sufficiently large values of $n$.  
For completeness, we include the following definition here.  
For a full treatment, see \cite{ec}.  

\begin{defn}
A map $f:\NN \to \NN$ is a \emph{quasipolynomial}
if $f(x) = a_d(x)x^d + \cdots + a_0(x)$, 
where each $a_i(x)$ is periodic with period $s_i$ and $a_d$ is not identically $0$.  
We call $d$ the \emph{degree} of $f$, and $s = \lcm(s_0, \ldots, s_d)$ the \emph{period} of $f$.  
A quasipolynomial $f$ is \emph{quasilinear} if it has degree 1.  
A function $g:\NN \to \NN$ is \emph{eventually quasilinear} 
if there exists a quasilinear function $f$ and $N \in \NN$ such that
$g(n) = f(n)$ for all $n > N$.  
\end{defn}

%%%%%%%%%%%%%%%%%%%%%%%%%%%%%%%%%%%%%%%%%%%%%%%%%%%%%%%%%%%%%%%%%%%%%%%%%
\section{Main Result}\label{s:quasi}%%%%%%%%%%%%%%%%%%%%%%%%%%%%%%%%%%%%%
%raggedbottom%%%%%%%%%%%%%%%%%%%%%%%%%%%%%%%%%%%%%%%%%%%%%%%%%%%%%%%%%%%%

In this section, we begin by explicitly stating the relationship between bullets for $n \in \Gamma$ and the value of $\omega(n)$ (Proposition~\ref{p:omega}) and introducing the notion of cover maps (Definition~\ref{d:covermap}).  We then apply this to prove the eventual quasilinearity of the $\omega$-function for any numerical monoid (Theorem~\ref{t:quasi}).  We conclude the section by providing an asymptotic description of the $\omega$-function (Corollary~\ref{c:growth}).

\begin{example}\label{e:3gen}
Let $\Gamma = \<6,9,20\>$.  Below is a plot of the values of the 
$\omega$-function for $\<6,9,20\>$ (left) and $\<11,13,15\>$ (right).  
Theorem~\ref{t:quasi} states that for large values, 
these graphs will look like a collection of discrete lines 
of slope $\frac{1}{6}$ and $\frac{1}{11}$, respectively.  

\begin{center}
\includegraphics[width=2.5in]{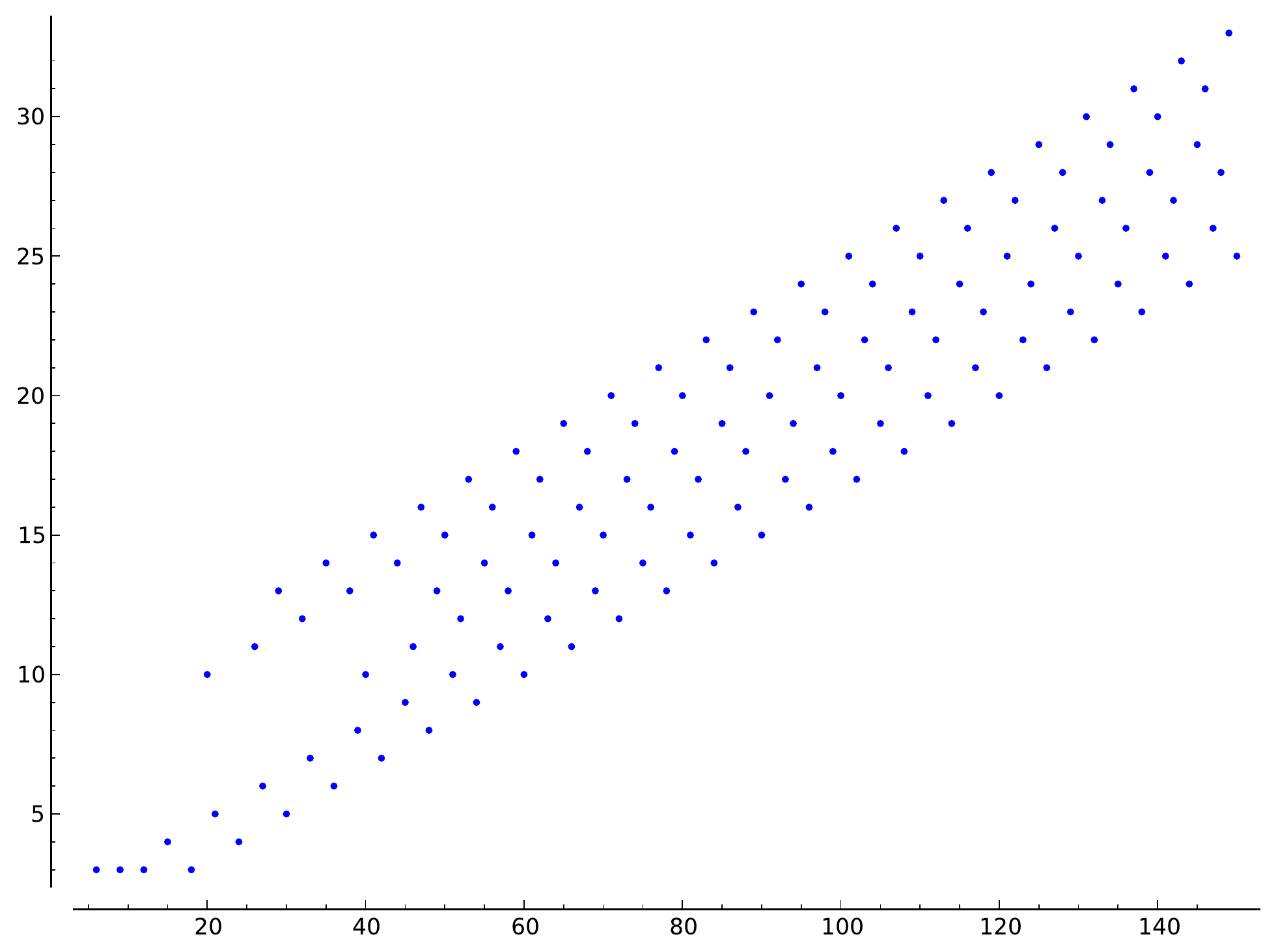}
\hspace{0.5in}
\includegraphics[width=2.5in]{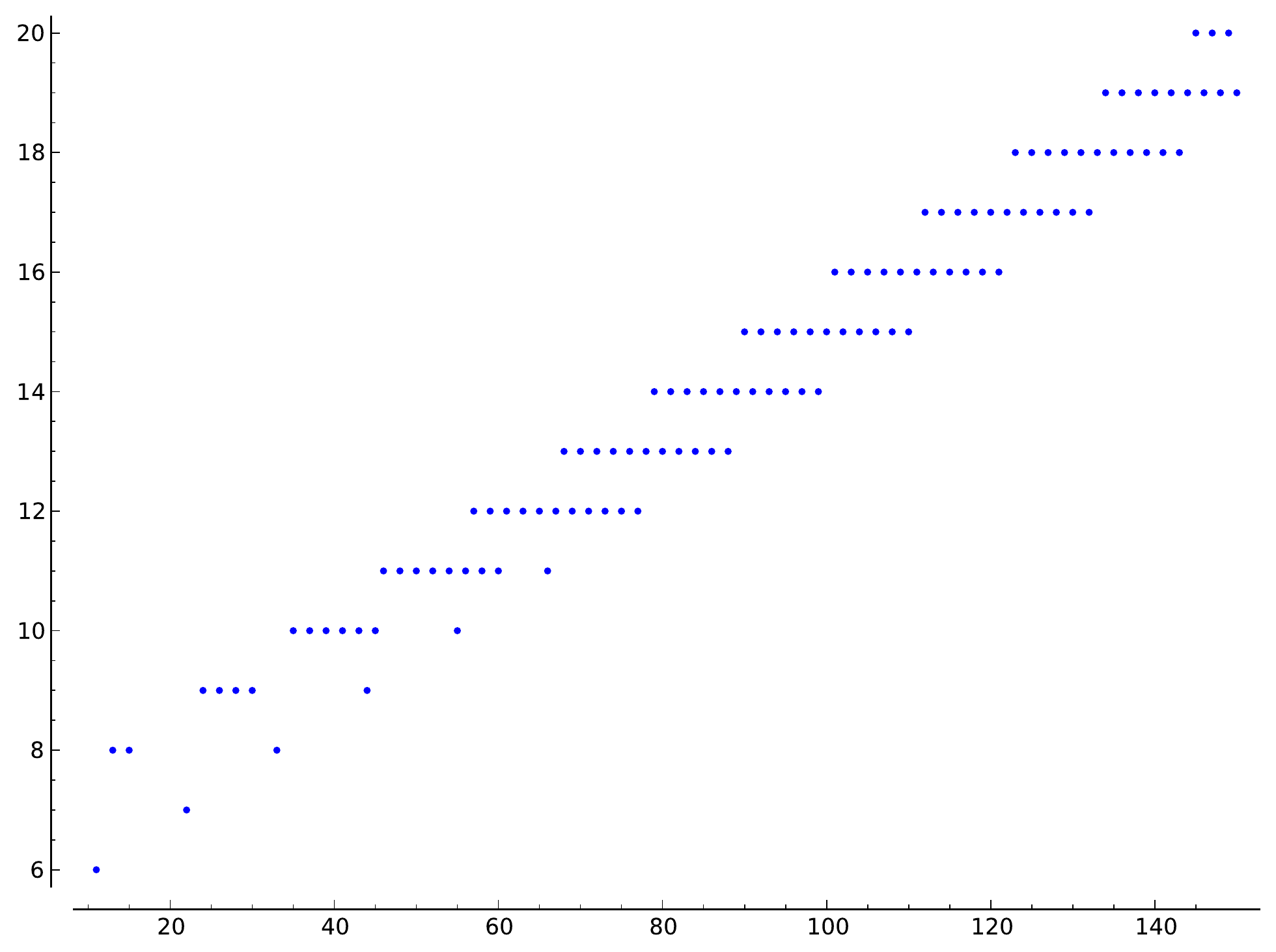}
\end{center}
\end{example}

\begin{prop}\label{p:omega}
For a numerical monoid $\Gamma$, $$\omega(n) = \max\{|\vec a| : \vec a \in \bul(n)\}$$ for all $n \in \Gamma$.  
\end{prop}

\begin{proof}
For any bullet $\vec a \in \bul(n)$, $\omega(n) \ge |\vec a|$, 
with equality exactly when $\vec a$ is a maximal bullet for $n$.  
Lemma~\ref{l:bullet} completes the proof.  
\end{proof}

\begin{remark}
Proposition~\ref{p:omega}, together with Lemma~\ref{l:quasi}b, 
gives an algorithm for computing $\omega(n)$.  
In particular, Lemma~\ref{l:quasi}b shows that 
every $n$ has a bullet with support $\{n_j\}$ for each $j$, say with length $b_j$.  
This means every other bullet $\vec a \in \bul(n)$ must satisfy
$a_j < b_j$.  This gives a bounded region in which to find $\bul(n)$, 
which by Proposition~\ref{p:omega} allows us to compute $\omega(n)$.  
This is essentially the algorithm given in \cite[The Omega Algorithm]{andalg}.  
\end{remark}

\begin{defn}\label{d:covermap}
Let $\Gamma = \<n_1, \ldots, n_k\>$ be a numerical monoid.  
For $n \in \Gamma$ and $a \in \NN$, the \emph{$j$-th cover map} 
is the map $\bul_j(n) \to \bul(n + an_j)$ given by 
$\phi_j(a_1, \ldots, a_k) = (a_1, \ldots, a_j + a, \ldots, a_k)$.  
\end{defn}

% \begin{remark}\label{r:covermap}
Each cover map $\bul_j(n) \to \bul(n + an_j)$ in Definition~\ref{d:covermap} 
is injective, and increases the length of a given factorization by $a$.  
By Proposition~\ref{p:omega}, this shows $\omega(n) + 1$ 
is a lower bound for $\omega(n + n_j)$ when $\bul_j(n) \neq \varnothing $.

Cover maps are, in many ways, the key to the proof of Theorem~\ref{t:quasi}, 
as they allow us to express the bullets that occur for large elements of $\Gamma$
in terms of the bullets of smaller elements in $\Gamma$.  
% \end{remark}

\begin{lemma}\label{l:quasi}
Suppose $\Gamma = \<n_1, \ldots, n_k\>$ is a minimally generated numerical monoid.  
\begin{enumerate}
\item[(a)] 
Fix $i < j$, and fix $\vec a \in \bul(n)$ with $a_i, a_j > 0$.  
Fix an $m \in \NN$ divisible by $n_i$ and $n_j$, and let
$\phi_i:\bul_i(n) \to \bul(n + mn_i)$ and $\phi_j:\bul_j(n) \to \bul(n + mn_j)$
denote the cover maps.  
We have $|\phi_i(\vec a)| \ge |\phi_j(\vec a)|$
with equality if and only if $n_i = n_j$.  

\item[(b)] 
For each $j \le k$, there exists a bullet $\vec b \in \bul(n)$ for $n$ with $\supp(\vec b) = \{n_j\}$.  

\end{enumerate}
\end{lemma}

\begin{proof}\text{ }
\begin{enumerate}
\item[(a)] 
Since $n_i \le n_j$, we see $|\phi_i(\vec a)| = |\vec a| + \frac{m}{n_i} \ge |\vec a| + \frac{m}{n_j} = |\phi_j(\vec a)|$.  

\item[(b)] 
When $b_j$ is sufficiently large, $b_jn_j - n \in \Gamma$.  
Letting $b_j$ be the minimal such value, we get $\vec b \in \bul(n)$.  
\end{enumerate}
\end{proof}

We are now ready to state and prove the main result.  

\begin{thm}\label{t:quasi}
Let $\Gamma = \<n_1, \ldots, n_k\>$ be a numerical monoid.  
For $n$ sufficiently large, $\omega(n)$ is quasilinear with period dividing $n_1$.  
In particular, there exists an explicit $N_0$ such that $\omega(n + n_1) = \omega(n) + 1$ for $n > N_0$.  
\end{thm}

\begin{proof}
Let $m = \lcm(n_1, \ldots, n_k)$, and let $n_0 = m \cdot \sum_{i=1}^k n_i$.  For $0 \le r < m$, let
$$c_r = \max\{|\vec b| - |\vec a| : \vec a, \vec b \in \bul(n_0 + r)\}$$
and let $c = \max\{c_r : 0 \leq r < m\}$.   Let $N_0 = n_0 + cm$.   We will show that
$\omega(n + n_1) = \omega(n) + 1$
for all $n > N_0$.  

Fix $n \in \Gamma$ with $n > n_0 + cm$, and write $n - n_0 = qm + r$ with $0 \le r < m$.  
Let $\phi_j:\bul_j(n) \to \bul(n + m)$ and $\psi_j:\bul(n_0 + r) \to \bul(n)$ 
denote the maps given by 
$$\phi_j(a_1, \ldots, a_k) = (a_1, \ldots, a_j + \tfrac{m}{n_j}, \ldots, a_k) \text{ and}$$ 
$$\psi_j(a_1, \ldots, a_k) = (a_1, \ldots, a_j + q \cdot \tfrac{m}{n_j}, \ldots, a_k)$$
respectively.  Notice that since $n \ge n_0$, each $(a_1, \ldots, a_k) \in \bul(n + m)$ has $a_j > m$ for some $j$, 
so 
$$\bigcup_{j=1}^k \image(\phi_j) = \bul(n + m)$$
that is, the union of the images of the cover maps $\phi_j$ cover $\bul(n + m)$.  
Also, since each $\psi_j$ is a composition of such cover maps, 
the images of the cover maps $\psi_j$ cover $\bul(n)$ as well.  

We first show there exists a maximal bullet for $n$ with nonzero first coordinate.  
Since every bullet in $\bul(n)$ is in the image of some $\psi_j$, 
fix $\vec a \in \bul(n_0 + r)$ with $a_1 = 0$ and some $a_j > 0$.  
By Lemma~\ref{l:quasi}b, there exists a bullet $\vec b \in \bul(n_0 + r)$ with $b_1 > 0$.  
By assumption we have $|\vec a| \le |\vec b| + c_r$.  
Each $\psi_i$ is the composition of cover maps
$$\bul_i(n_0 + r) \to \bul_i(n_0 + r + m) \to \cdots \to \bul_i(n_0 + r + qm).$$
Since $n_1 < n_j$, the images of $\vec a$ and $\vec b$ in $\bul(n_0 + r + \ell m)$ 
differ in length by at most $c_r - \ell$.  
Since $q \ge c_r$, this gives $|\psi_1(\vec b)| \ge |\psi_j(\vec a)|$.  
% by Lemma~\ref{l:quasi}a 

Next, we will show that $\omega(n + m) = \omega(n) + \frac{m}{n_1}$.  
Fix a maximal bullet $\vec b \in \bul(n)$ with $b_1 > 0$.  
We claim $\phi_1(\vec b) \in \bul(n + m)$ is maximal as well.  
Every bullet for $n + m$ is the image of some $\phi_j$.  
By Lemma~\ref{l:quasi}a, $\phi_1$ gives the largest increase in length, 
so since $\vec b$ has maximal length in $\bul(n)$, it follows that $\phi_1(\vec b)$
has maximal length in $\bul(n + m)$.  This yields 
$$\omega(n + m) = |\phi_1(\vec b)| = |\vec b| + \tfrac{m}{n_1} = \omega(n) + \tfrac{m}{n_1}$$
as desired.  

Finally, we will show that $\omega(n + n_1) = \omega(n) + 1$.  
Notice the map $\phi_1$ is given by the composition
$$\bul(n) \to \bul(n + n_1) \to \cdots \to \bul(n + \frac{m}{n_1}n_1)$$
and each map is injective.  By Lemma~\ref{l:quasi}a this means 
$$\omega(n) \le \omega(n + n_1) - 1 \le \cdots \le \omega(n + \frac{m}{n_1}n_1) - \frac{m}{n_1}$$
However, the first and last terms are equal, so in fact $\omega(n + n_1) = \omega(n) + 1$.  
This completes the proof.  
\end{proof}

\begin{remark}\label{r:bound}
Fix a numerical monoid $\Gamma$, and let $f$ denote the quasipolynomial 
such that $f(n) = \omega(n)$ for sufficiently large $n$.  
The \emph{dissonance point} of the $\omega$-function is the maximal $n$ such that $\omega(n) \ne f(n)$.  
The proof of Theorem~\ref{t:quasi} indicates that $N_0 = n_0 + cm$ is an upper bound for the dissonance point, 
where $m = \lcm(n_1, \ldots, n_k)$, $n_0 = m \cdot \sum_{i=1}^k n_i$ and $c = \max\{c_r : 0 \leq r < m\}$, 
where 
$$c_r = \max\{|\vec b| - |\vec a| : \vec a, \vec b \in \bul(n_0 + r)\}$$
for $0 \le r < m$.  
% Although we give an explicit bound for the dissonance point in the proof of Theorem~\ref{t:quasi}, 
The actual value of the dissonance point seems to vary drastically for different monoids.  
For instance, for the numerical monoid $\<10,11,15\>$, the dissonance point is 380, 
but for the numerical monoid $\<11,13,15\>$ from Example~\ref{e:3gen}, the dissonance point is 66.  
For both of these monoids, $N_0 > 10000$ and thus is significantly larger than the actual dissonance points.  

% Even when $\Gamma$ has embedding dimension 3, the dissonance point can occur 
% just past the Frobenius number of $\Gamma$, or many times larger than the product of the generators.  
\end{remark}

We now give the following corollaries to Theorem~\ref{t:quasi}.  The first is a consequence 
of the proof of Theorem~\ref{t:quasi} and strengthens Lemma~\ref{l:quasi}b.  
The second was proved in \cite[Theorem 4.9]{andalg} for embedding dimension~2 but is generalized here to arbitrary embedding dimension.  

\begin{cor}\label{c:bul1}
For $n \in \Gamma$ sufficiently large, there exists a maximal bullet $\vec a \in \bul(n)$ with $a_1 > 0$.  
\end{cor}

\begin{cor}\label{c:growth}
We have 
$$\lim_{n \to \infty} \frac{\omega(n)}{n} = \frac{1}{n_1}$$
for any numerical monoid $\Gamma$.  
\end{cor}

\begin{proof}
By Theorem~\ref{t:quasi}, we can eventually write $\omega(n) = \frac{1}{n_1}n + a(n)$, 
where $a(n)$ is periodic with period at most $n_1$.  In particular, 
$a(n)$ takes on finitely many values, and thus is bounded.  
Therefore, as $n$ grows large, $\frac{a(n)}{n} \to 0$, so $\frac{\omega(n)}{n} \to \frac{1}{n_1}$, 
% $$\frac{\omega(n)}{n} = \frac{1}{n_1} + \frac{a(n)}{n} \rightsquigarrow \frac{1}{n_1},$$
as desired.  
\end{proof}

%%%%%%%%%%%%%%%%%%%%%%%%%%%%%%%%%%%%%%%%%%%%%%%%%%%%%%%%%%%%%%%%%%%%%%%%%
\section{Embedding Dimension 2}\label{s:2gen}%%%%%%%%%%%%%%%%%%%%%%%%%%%%
%raggedbottom%%%%%%%%%%%%%%%%%%%%%%%%%%%%%%%%%%%%%%%%%%%%%%%%%%%%%%%%%%%%

In this section, we utilize the above results about bullets and the long-term behavior of the $\omega$-function to provide an explicit formula for $\omega(n)$ when $\Gamma$ has embedding dimension $2$.

\begin{example}\label{e:2gen}
Let $\Gamma = \<3,7\>$.  The following gives the omega value and maximal bullets 
for elements of $\Gamma$ between 3 and 20.  

\begin{center}
\begin{tabular}{|l||l|l|l||l|l|l||l|l|l|}
\hline
$q$ & $3q$ & $\omega$ & $\mbul$ & $3q+1$ & $\omega$ & $\mbul$ & $3q+2$ & $\omega$ & $\mbul$ \\
\hline
1 & 3 & 3 & $(0,3)$ & 4 &  &  & 5 &  &  \\
% \hline
2 & 6 & 3 & $(0,3)$ & 7 & 7 & $(7,0)$ & 8 &  &  \\
% \hline
3 & 9 & 3 & $(3,0),(0,3)$ & 10 & 8 & $(8,0)$ & 11 &  &  \\
% \hline
4 & 12 & 4 & $(4,0)$ & 13 & 9 & $(9,0)$ & 14 & 7 & $(7,0)$ \\
% \hline
5 & 15 & 5 & $(5,0)$ & 16 & 10 & $(10,0)$ & 17 & 8 & $(8,0)$ \\
% \hline
6 & 18 & 6 & $(6,0)$ & 19 & 11 & $(11,0)$ & 20 & 9 & $(9,0)$ \\
% \hline
% \vdots & \vdots & \vdots & \vdots & \vdots & \vdots & \vdots & \vdots & \vdots & \vdots \\
\hline
\end{tabular}
\end{center}
\end{example}

Recall that in a numerical monoid $\Gamma$ with $n \in \Gamma$, the Ap\'ery set of $n$ is defined as $$\Ap(\Gamma,n) = \{m \in \Gamma : m-n \not \in \Gamma\}.$$   See~\cite{numerical} for a full treatment of Ap\'ery sets.  We include the following lemma characterizing the Ap\'ery sets of the generators of a numerical monoid with embedding dimension $2$. 

\begin{lemma}\label{l:2apery}
Let $\Gamma = \<n_1, n_2\>$ be a numerical monoid.  
Then 
$$\Ap(\Gamma,n_i) = \{an_j : 0 \le a < n_i\}$$
for $\{i,j\} = \{1,2\}$.  
\end{lemma}

\begin{proof}
Since $n_1$ and $n_2$ are relatively prime, $an_j \in \Ap(\Gamma,n_i)$ for $0 \le a < n_i$.  This gives $n_i$ distinct elements of $\Ap(\Gamma,n_i)$, and by definition $|\Ap(\Gamma,n_i)| = n_i$.  
\end{proof}

The following lemma appeared as \cite[Lemma 4.3]{andalg}.  
We state and prove it here using the terminology of this paper.  

\begin{lemma}\label{l:2pbul}
Let $\Gamma = \<n_1, n_2\>$ be a numerical monoid.  
For $n \in \Gamma$, we have 
$$\bul(n) = \{(u,0),(0,v)\} \cup \{(a_1,a_2) : n = a_1n_1 + a_2n_2\}$$
for some $u, v > 0$.  
Moreover, only $(u,0)$ and $(0,v)$ can be maximal bullets for $n$.  
% is a subset of $\{(\omega(n), 0), (0, \omega(n))\}$.  
\end{lemma}

\begin{proof}
Clearly $(u,0), (0,v) \in \bul(n)$ for some $u, v > 0$, 
so fix a bullet $(a_1, a_2) \in \bul(n)$ with $a_1, a_2 > 0$.  
Let $b = a_1n_1 + a_2n_2 - n$.  
Then $b \in \Gamma$ but $b - n_1, b - n_2 \notin \Gamma$.  
This means $b \in \Ap(\Gamma,n_1) \cap \Ap(\Gamma,n_2)$.  
But this only happens when $b = 0$, proving the first claim.  

Now, if $n = un_1$, then $un_1$ is the longest factorization of $n$, 
so the second claim follows.  
Otherwise, we have $un_1 - n \in \Ap(\Gamma,n_1)$, so 
$un_1 - n = an_2$ for some nonnegative $a < n_1$.  
Then $n + an_2 = un_1$, so for any factorization $(a_1,a_2)$ of $n$, 
$a_1 + a_2 + a \le u$, meaning $|(u,0)| > |(a_1,a_2)|$.  
This completes the proof.  
\end{proof}

We now provide a formula for the value of $\omega(n)$ in a numerical monoid of embedding dimension $2$ for $n$ sufficiently large.  

\begin{thm}\label{t:2omega}
Let $\Gamma = \<n_1, n_2\>$ be a numerical monoid.  
Fix $n$ sufficiently large.  
Write $n = qn_1 + r$ with $0 \le r < n_1$, 
and let $a \ge 0$ be minimal such that $an_1 \equiv r \bmod n_2$.  
Then we have $\omega(n) = q + a$.  
\end{thm}

\begin{proof}
% m = beginning of each column?
For any $n \in \Gamma$, let $u(n)$ and $v(n)$ denote the values 
such that $(u(n),0), (0,v(n)) \in \bul(n)$.  
By applying covering maps, we see that 
$u(n + n_1n_2) = u(n) + n_2$ and $v(n + n_1n_2) = v(n) + n_1$.  
Let 
$c = \max(\{0\} \cup \{v(n_1n_2 + t) - u(n_1n_2 + t) : 0 \le t < n_1n_2\}),$
fix $n > cn_1n_2$, and write $n = sn_1n_2 + t$ with $0 \le t < n_1n_2$.  
Then 
\begin{eqnarray*}
u(n) - v(n) &=& u(sn_1n_2 + t) - v(sn_1n_2 + t) \\
&=& (s - 1)(n_2 - n_1) + u(n_1n_2 + t) - v(n_1n_2 + t) \\
&\ge& c(n_2 - n_1) - (v(n_1n_2 + t) - u(n_1n_2 + t)) \ge c(n_2 - n_1) - c \ge 0
\end{eqnarray*}
so by Lemma~\ref{l:2pbul}, $(u(n),0)$ is a maximal bullet for $n$, meaning $\omega(n) = u$.  
Write $n = qn_1 + r$ with $0 \le r < n_1$, 
and let $a \ge 0$ be minimal such that $an_1 \equiv r \bmod n_2$.  
Then  
$$(a + q)n_1 - n = an_1 + qn_1 - n = an_1 - r \in \<n_2\>$$
and by minimality of $a$, $(a + q - 1)n_1 - n = (an_1 - r) - n_1 \notin \Gamma$, 
so in fact $u = a + q$.  
\end{proof}

\begin{remark}
The authors of \cite{andalg} also provide a formula for the $\omega$-values in a numerical monoid of embedding dimension $2$.  Their formula requires an explicit computation for each $n \in \Gamma$, whereas Theorem~\ref{t:2omega} only requires $n_1$ computations to find all $\omega$-values for large $n$.
\end{remark}

\begin{remark}\label{r:ed2bound}
Theorem~\ref{t:2omega} improves the bound on the dissonance point for 
2-generated numerical monoids to $cn_1n_2$ with
$$c = \max(\{0\} \cup \{v(n_1n_2 + t) - u(n_1n_2 + t) : 0 \le t < n_1n_2\})$$
where $u(n)$ and $v(n)$ denote the unique values 
such that $(u(n),0), (0,v(n)) \in \bul(n)$, as guaranteed by Lemma~\ref{l:2pbul}.  
\end{remark}

\begin{example}
Let $\Gamma = \<3,7\>$ as defined in Example~\ref{e:2gen}.  
For $n \in \Gamma$ and writing $n = 3q + r$ for $0 \le r < 3$, 
Theorem~\ref{t:2omega} allows us to write $\omega(n) = q + a(r)$, where 
$$\omega(n) = \left\{\begin{array}{cc}
q + 0 & \text{ for } r = 0 \\
q + 5 & \text{ for } r = 1 \\
q + 3 & \text{ for } r = 2
\end{array}\right.$$
and by the table in Example~\ref{e:2gen}, 
we see this holds for $n \ge 7$.  
\end{example}

%\begin{example}
%Lower bound for 2gen is tight-ish.  $\Gamma = \<4,5\>$, stabalization at $n = 44$.  
%\end{example}

%%%%%%%%%%%%%%%%%%%%%%%%%%%%%%%%%%%%%%%%%%%%%%%%%%%%%%%%%%%%%%%%%%%%%%%%%
\section{Future Work}\label{s:futurework}%%%%%%%%%%%%%%%%%%%%%%%%%%%%%%%%
%raggedbottom%%%%%%%%%%%%%%%%%%%%%%%%%%%%%%%%%%%%%%%%%%%%%%%%%%%%%%%%%%%%

% \subsection*{Period}
In all of the examples of $\Gamma$ we have computed, 
the eventual period of the $\omega$-function is the smallest generator.  
We record this here.  

\begin{conj}
For any numerical monoid $\Gamma$ with least generator $n_1$.  
For $n$ sufficiently large, $\omega(n)$ is quasilinear with period exactly $n_1$.  
\end{conj}

Another question concerns a more precise bound on the dissonance point.  

\begin{prob}
Fix a numerical monoid $\Gamma$, and let $f$ denote the quasipolynomial 
such that $f(n) = \omega(n)$ for sufficiently large $n$.  
Characterize the maximal $n$ such that $\omega(n) \ne f(n)$.  
\end{prob}

% \subsection*{Maximal Bullets}

Corollary~\ref{c:bul1} states that each sufficiently large $n$ 
has a maximal bullet whose support contains $n_1$.  
Thus, a related question involves characterizing the maximal bullets of large $n \in \Gamma$.  
In many examples (for instance, when $\Gamma$ have embedding dimension 2), 
for large $n$, the bullet with support $\{n_1\}$ is maximal.  
However, this is not always the case.  For instance, when $\Gamma = \<4,5,6\>$ and
$n$ is sufficiently large with $n \equiv 0 \bmod 4$, every maximal bullet $\vec a \in \mbul(n)$
satisfies $a_2 = 2$.  

\begin{prob}
Characterize the maximal bullets for numerical semigroups $\Gamma$ 
with embedding dimension greater than 2.  
\end{prob}

% \subsection*{Hilbert Polynomials}
Our initial efforts to prove Theorem~\ref{t:quasi} 
involved attempts to construct a graded ring $R$ and a graded $R$-module $M$ 
whose Hilbert function $\mathcal H(M;-)$ takes on the values of the $\omega$-function.  
Prior to finding such a construction, the direct proof given in this paper was obtained.  
% Unfortunately, these attempts proved futile, and eventually 
% we found the direct proof given in this paper.  
If a proof using Hilbert functions were found, the resulting construction would give 
a combinatorial algebraic interpretation of the $\omega$-function.  
See \cite{cca} for more on the Hilbert function.  

% The fact that the $\omega$-function is eventually a quasipolynomial raises the following question.  

\begin{prob}\label{p:hilbert}
Given $\Gamma$, can one construct a graded ring $R$ and graded $R$-module $M$
(both depending on $\Gamma$) for which $\mathcal H(M;n) = \omega(n)$ for $n \in \Gamma$?  
\end{prob}

\section{Acknowledgements}

Much of the work was completed during the Pacific Undergraduate Research Experience in Mathematics (PURE Math), which was funded by National Science Foundation grants DMS-1045147 and DMS-1045082 and a supplementary grant from the National Security Agency.  The authors would like to thank Scott Chapman and Ezra Miller for numerous helpful conversations,  as well as Emelie Curl, Staci Gleen, and Katrina Quinata for their initial observations on $\omega$-functions.
 % and their love of Chicken McNuggets.

%%%%%%%%%%%%%%%%%%%%%%%%%%%%%%%%%%%%%%%%%%%%%%%%%%%%%%%%%%%%%%%%%%%%%%%%%
%%%%%%%%%%%%%%%%%%%%%%%%%%%%%%%%%%%%%%%%%%%%%%%%%%%%
%%%%%%%%%%%%%%%%%%%%%%%%%%%%%%%%%%%%%%%%%%%%%%%%%%%%%%%%%%%%%%%%%%%%%%%%%

%%%%%%%%%%%%%%%%%%%%%%%%%%%%%%%%%%%%%%%%%%%%%%%%%%%%%%%%%%%%%%%%%%%%%%%%%
\end{document}